\documentclass[11pt]{article}

\usepackage{amsmath, amssymb}
\usepackage{authblk}
\usepackage{geometry}
\usepackage{hyperref}
\usepackage{graphicx}
\usepackage{epstopdf}
\usepackage{mathrsfs}
\usepackage{amsfonts}
\usepackage{amscd,amsthm,bm,psfrag}
\usepackage[numbers,sort&compress]{natbib}

\makeatletter
\renewcommand{\@seccntformat}[1]{{\csname the#1\endcsname}{\normalsize.}\hspace{.5em}}

\makeatother

\numberwithin{equation}{section}

\def \[{\begin{equation}}
\def \]{\end{equation}}

\newtheorem{thm}{Theorem}[section]

\newtheorem{lem}[thm]{Lemma}
\newtheorem{cor}[thm]{Corollary}
\newtheorem{prop}[thm]{Proposition}

\newtheorem{cla}{Claim}

\newtheorem*{thm*}{Theorem}
\newtheorem*{prop*}{Proposition}

\newenvironment{kst}
{\setlength{\leftmargini}{1.3\parindent}
 \begin{itemize}
 \setlength{\itemsep}{-1.1mm}}
{\end{itemize}}

\begin{document}
\title{On Tur\'{a}n problems for Berge forests
}

\author[a,b]{Junpeng Zhou\,}
\author[c,*]{D\'{a}niel Gerbner\,}
\author[a,b]{Xiying Yuan\,}

\affil[a]{\small \,Department of Mathematics, Shanghai University, Shanghai 200444, PR China}
\affil[b]{\small \,Newtouch Center for Mathematics of Shanghai University, Shanghai 200444, PR China}
\affil[c]{\small \,Alfr\'ed R\'enyi Institute of Mathematics, HUN-REN}

\date{}

\maketitle

\footnotetext{*\textit{Corresponding author}. Email addresses: \texttt{junpengzhou@shu.edu.cn} (J. Zhou), \texttt{gerbner.daniel@renyi.hu} (D. Gerbner), \texttt{xiyingyuan@shu.edu.cn} (X. Yuan).}




\begin{abstract}
For a graph $F$, an $r$-uniform hypergraph $H$ is a Berge-$F$ if there is a bijection $\phi:E(F)\rightarrow E(H)$ such that $e\subseteq \phi(e)$ for each $e\in E(F)$. Given a family $\mathcal{F}$ of $r$-uniform hypergraphs, an $r$-uniform hypergraph is $\mathcal{F}$-free if it does not contain any member in $\mathcal{F}$ as a subhypergraph. The Tur\'{a}n number 
of $\mathcal{F}$ is the maximum number of hyperedges in an $\mathcal{F}$-free $r$-uniform hypergraph on $n$ vertices. 
In this paper, some exact and general results on the Tur\'{a}n numbers for several types of Berge forests are obtained. 
\end{abstract}

{\noindent{\bf Keywords}: Tur\'{a}n number, Berge hypergraph, forest}

{\noindent{\bf AMS subject classifications:} 05C35, 05C65}

\section{\normalsize Introduction} 
A \textit{hypergraph} $H=(V(H),E(H))$ consists of a vertex set $V(H)$ and a hyperedge set $E(H)$, where each hyperedge in $E(H)$ is a nonempty subset of $V(H)$. If $|e|=r$ for every $e\in E(H)$, then $H$ is called an \textit{$r$-uniform hypergraph} ($r$-graph for short). For simplicity, let $e(H):=|E(H)|$. The \textit{degree} $d_H(v)$ of a vertex $v$ is the number of hyperedges containing $v$ in $H$.

Let $\mathcal{F}$ be a family of $r$-graphs. An $r$-graph $H$ is called \textit{$\mathcal{F}$-free} if $H$ does not contain any member in $\mathcal{F}$ as a subhypergraph. The \textit{Tur\'{a}n number} ${\rm{ex}}_r(n,\mathcal{F})$ of $\mathcal{F}$ is the maximum number of hyperedges in an $\mathcal{F}$-free $r$-graph on $n$ vertices. If $\mathcal{F}=\{G\}$, then we write ${\rm{ex}}_r(n,G)$ instead of ${\rm{ex}}_r(n,\{G\})$. When $r=2$, we write $\mathrm{ex}(n,\mathcal{F})$ instead of $\mathrm{ex}_2(n,\mathcal{F})$.

Let $F$ be a graph. An $r$-graph $H$ is a \textit{Berge-$F$} if there is a bijection $\phi: E(F)\rightarrow E(H)$ such that $e\subseteq \phi(e)$ for each $e\in E(F)$. The graph $F$ is called a \textit{skeleton} of $H$.
Note that the word \textit{core} is also used in the literature. For a fixed graph $F$ there are many hypergraphs that are a Berge-$F$. For convenience, we refer to this collection of hypergraphs as ``Berge-$F$''. 
Berge \cite{A2} defined the Berge cycle, and Gy\H{o}ri, Katona and Lemons \cite{B6} defined the Berge path. Later, Gerbner and Palmer \cite{B3} generalized the established concepts of Berge cycle and Berge path to general graphs.

Hypergraph Tur\'{a}n problems are the central topics of extremal combinatorics. In particular, Tur\'{a}n problems on Berge hypergraphs have been extensively studied, yielding numerous related extremal results. 
Gy\H{o}ri \cite{Gyo} showed that for $r=3,4$, an $n$-vertex Berge triangle-free $r$-graph has at most $\lfloor n^2/8(r-2)\rfloor$ hyperedges if $n$ is large enough, and this bound is sharp. Gy\H{o}ri and Lemons \cite{B2} showed that the Tur\'{a}n numbers of Berge-$C_{2k}$ and Berge-$C_{2k+1}$ have an order of magnitude of $O(n^{1+1/k})$. 
Gy\H{o}ri, Katona and Lemons \cite{B6} generalized the Erd\H{o}s-Gallai theorem to Berge paths. Specifically, they determined ${\rm{ex}}_r(n,{\rm Berge}{\text -}P_\ell)$ for the cases when $\ell>r+1>3$ and $r\geq\ell>2$. The case when $\ell=r+1>2$ was settled by Davoodi, Gy\H{o}ri, Methuku and Tompkins \cite{10-10A1}. 
\begin{thm}[Gy\H{o}ri, Katona and Lemons \cite{B6}, Davoodi, Gy\H{o}ri, Methuku and Tompkins \cite{10-10A1}]\label{lem2.2}
\ \ 
\begin{kst}
\item[\textbf{(i)}] If $\ell\geq r+1>3$, then ${\rm{ex}}_r(n,{\rm Berge}{\text -}P_\ell)\leq \frac{n}{\ell}\binom{\ell}{r}$. Furthermore, this bound is sharp whenever $\ell$ divides $n$.
\item[\textbf{(ii)}] If $r\geq\ell>2$, then ${\rm{ex}}_r(n,{\rm Berge}{\text -}P_\ell)\leq \frac{n(\ell-1)}{r+1}$. Furthermore, this bound is sharp whenever $r+1$ divides $n$.
\end{kst}
\end{thm}

Analogous to the graph case, the connected version of this problem is also studied. A hypergraph $H$ is said to be \textit{connected} if for every two vertices there is a Berge path containing both of them. Let $\mathcal{F}$ be a family of $r$-graphs. Denote by ${\rm{ex}}^{\rm con}_r(n,\mathcal{F})$ the maximum number of hyperedges in an $n$-vertex connected $\mathcal{F}$-free $r$-graph. Gy\H{o}ri, Methuku, Salia, Tompkins and Vizer \cite{GMSTV} determined the asymptotics of ${\rm{ex}}^{\rm con}_r(n,{\rm Berge}{\text -}P_\ell)$. Later, F\"{u}redi, Kostochka and Luo \cite{FKL2} determined ${\rm{ex}}^{\rm con}_r(n,{\rm Berge}{\text -}P_\ell)$ for sufficiently large $n$ and $\ell\geq 4r \geq 12$. Independently, Gy\H{o}ri, Salia and Zamora \cite{GSZ} determined ${\rm{ex}}^{\rm con}_r(n,{\rm Berge}{\text -}P_\ell)$ for sufficiently large $n$ and $\ell\geq 2r+13 \geq 18$. 

\begin{thm}[Gy\H{o}ri, Salia and Zamora \cite{GSZ}]\label{lem4.3}
For all integers $n$, $\ell$ and $r$ there exists an $N_{\ell,r}$ such that for $n > N_{\ell,r}$ and $\ell \geq 2r + 13 \geq 18$,
\begin{eqnarray*}
{\rm{ex}}^{\rm con}_r(n,{\rm Berge}{\text -}P_\ell)=
\binom{\frac{\ell-1}{2}}{r-1}\left( n - \left\lfloor \frac{\ell-1}{2} \right\rfloor \right)
+ \binom{\left\lfloor \frac{\ell-1}{2} \right\rfloor}{r}
+ \mathbf{1}_{2 \mid \ell} \binom{\left\lfloor \frac{\ell-1}{2} \right\rfloor}{r-2},
\end{eqnarray*}
where $\mathbf{1}_{2 \mid \ell}=1$ if $2 \mid \ell$, and $\mathbf{1}_{2 \mid \ell}=0$ otherwise. 
\end{thm}

Furthermore, Gerbner, Nagy,  Patk\'os, Salia and Vizer \cite{GNPSV} proved a stability version of the above connected result. 
Results on the Turán numbers of Berge copies of graphs other than cycles can be found e.g. in \cite{B3,B5,ZYW,ZYC,B4,GGNPXZ,Ge2}. For a short survey on Tur\'{a}n problems on Berge hypergraphs, one can refer to Subsection 5.2.2 in \cite{A666}. 

There are fewer results known about Berge forests. Gerbner, Methuku and Palmer \cite{B4} established bounds on the Tur\'{a}n numbers of Berge trees. Gy\H{o}ri, Salia, Tompkins and Zamora \cite{10-10A2} showed that any $r$-graph with more than $\frac{n(k+1)}{r+1}$ hyperedges contains a Berge copy of any tree with $k$ edges different from the star, where $r\geq k(k-2)$.
Kang, Ni and Shan \cite{B7} studied the Tur\'{a}n numbers of Berge matchings for $r$-graph. Khormali and Palmer \cite{D4} proved an asymptotic result for the Tur\'{a}n numbers of Berge star forests. In particular, they determined the Tur\'{a}n numbers of Berge matchings for sufficiently large $n$. 

The purpose of this paper is to investigate Tur\'{a}n problems on Berge forests and to give some exact and general results. 





\section{\normalsize Main results} 
In this paper, we determine some exact results on the Tur\'{a}n numbers for several types of Berge forests. 
Let $H_1\cup H_2$ denote the disjoint union of hypergraphs $H_1$ and $H_2$, and $kH$ denote the disjoint union of $k$ hypergraphs $H$. For positive integers $a$ and $b$, define $\binom{a}{b}=0$ if $a<b$. 

\subsection{\normalsize Forests with star components} 
\ \ \ \ 
Let $T_\ell$ denote a tree with $\ell$ edges. In particular, let $P_\ell$ and $S_\ell$ denote a path and a star with $\ell$ edges, respectively.
The Erd\H os-S\'os conjecture \cite{Er} states that ${\rm ex}(n,T_\ell)\leq \frac{n(\ell-1)}{2}$. The conjecture is known to hold for paths by the Erd\H os-Gallai theorem \cite{EGa} and for spiders \cite{FHL}. First, for small $r$, we establish an exact result for a class of Berge forests that contains at least one star. 

\begin{thm}\label{thm1}
    Let integers $k\geq2$, $\ell\geq1$, $1\leq i\leq \ell+1$ and $2\leq r\leq k+\ell-1$. Suppose that ${\rm{ex}}_p(n,{\rm Berge}{\text -}T_\ell)\leq \binom{\ell}{p}\frac{n}{\ell}$ for each $2\le p\le r$, in particular, the Erd\H os-S\'os conjecture holds for $T_\ell$. Then for sufficiently large $n$, 
    \begin{eqnarray*}
    {\rm{ex}}_r(n,{\rm Berge}{\text -}T_\ell\cup (k-1)S_i)\leq \bigg(\binom{\ell+k-1}{r}-\binom{k-1}{r}\bigg)\bigg\lceil\frac{n-k+1}{\ell}\bigg\rceil+\binom{k-1}{r}. 
    \end{eqnarray*}
    Moreover, let $0\leq t\leq k-1$. If $\ell\,|\,n-k+1$, then 
    \begin{eqnarray*}
    {\rm{ex}}_r(n,{\rm Berge}{\text -}T_\ell\cup tS_\ell\cup (k-1-t)S_{\ell+1})= \bigg(\binom{\ell+k-1}{r}-\binom{k-1}{r}\bigg)\frac{n-k+1}{\ell}+\binom{k-1}{r}. 
    \end{eqnarray*}
\end{thm}

Let $M_{s+1}$ denote a matching of size $s+1$, i.e. the graph that consists of $s+1$ independent edges. Note that the condition on $T_\ell$ in the above theorem is trivial when $\ell=1$. Then the above theorem implies the following result. 

\begin{cor}\label{thmnew11}
    Let integers $k\geq2$, $2\leq r\leq k$ and $1\leq t\leq k$. Then for sufficiently large $n$,
    \begin{eqnarray*}
    {\rm{ex}}_r(n,{\rm Berge}{\text -}M_t\cup (k-t)S_2)= \binom{k-1}{r-1}(n-k+1)+\binom{k-1}{r}.
    \end{eqnarray*}
\end{cor}

Gerbner, Methuku and Palmer \cite{B4} showed that if the Erd\H os-S\'os conjecture holds for all subtrees of $T_\ell$ and $\ell>r+1>3$, then ${\rm{ex}}_r(n,{\rm Berge}{\text -}T_\ell)\leq \binom{\ell}{r}\frac{n}{\ell}$. Combining this result with Theorem \ref{thm1}, we derive the following proposition. 

\begin{prop}\label{newcor}
    Let integers $k\geq2$, $\ell\geq5$, $1\leq i\leq \ell+1$ and $2\leq r<\ell-1$. Suppose that the Erd\H os-S\'os conjecture holds for $T_\ell$ and its subtrees. Then for sufficiently large $n$,
    \begin{eqnarray*}
    {\rm{ex}}_r(n,{\rm Berge}{\text -}T_\ell\cup (k-1)S_i)\leq \bigg(\binom{\ell+k-1}{r}-\binom{k-1}{r}\bigg)\bigg\lceil\frac{n-k+1}{\ell}\bigg\rceil+\binom{k-1}{r}.
    \end{eqnarray*}
    Moreover, let $0\leq t\leq k-1$. If $\ell\,|\,n-k+1$, then 
    \begin{eqnarray*}
    {\rm{ex}}_r(n,{\rm Berge}{\text -}T_\ell\cup tS_\ell\cup (k-1-t)S_{\ell+1})= \bigg(\binom{\ell+k-1}{r}-\binom{k-1}{r}\bigg)\frac{n-k+1}{\ell}+\binom{k-1}{r}. 
    \end{eqnarray*}
\end{prop}

Since the Erd\H os-S\'os conjecture holds for paths, the above proposition implies sharp bounds for ${\rm{ex}}_r(n,{\rm Berge}{\text -}P_\ell\cup tS_\ell\cup (k-1-t)S_{\ell+1})$ when $3\leq r\leq\ell-1$, where $0\leq t\leq k-1$. 

Furthermore, we also obtain the following result for large $r$.

\begin{thm}\label{thm2}
    Let integers $k\geq2$, $\ell_1\geq3$, $\ell_2\geq2$ and $r\geq\ell_1+\ell_2+k-1$. Let $\ell_{\rm max}:=\max\{\ell_1,\ell_2\}$ and $\ell_{\rm min}:=\min\{\ell_1,\ell_2\}$. Then for sufficiently large $n$,
    \begin{eqnarray*}
    (\ell_{\rm min}-1)\Big\lfloor\frac{n-k+1}{r-k+2}\Big\rfloor\leq {\rm{ex}}_r(n,{\rm Berge}{\text -}P_{\ell_1}\cup (k-1)S_{\ell_2})\leq \frac{(\ell_{\rm max}-1)n}{r-k+2}+O(1).
    \end{eqnarray*}
\end{thm}

When $\ell_1=\ell_2$, Theorem \ref{thm2} yields the following corollary.

\begin{cor}\label{cor2}
    Let integers $k\geq2$, $\ell\geq3$ and $r\geq 2\ell+k-1$. For sufficiently large $n$,
    \begin{eqnarray*}
    {\rm{ex}}_r(n,{\rm Berge}{\text -}P_{\ell}\cup (k-1)S_{\ell})=\frac{(\ell-1)n}{r-k+2}+O(1).
    \end{eqnarray*}
\end{cor}



More generally, we establish a generalization of Theorem \ref{thm2} for larger $r$ and every tree. 

\begin{thm}\label{thm21}
    Let integers $k\geq2$, $\ell_1\geq3$, $\ell_2\geq2$ and $r\geq\max\{\ell_1(\ell_1-2),\ell_1+\ell_2+k-1\}$. Let $T_{\ell_1}\neq S_{\ell_1}$, $\ell_{\rm max}:=\max\{\ell_1,\ell_2\}$ and $\ell_{\rm min}:=\min\{\ell_1,\ell_2\}$. Then for sufficiently large $n$,
    \begin{eqnarray*}
    (\ell_{\rm min}-1)\Big\lfloor\frac{n-k+1}{r-k+2}\Big\rfloor\leq {\rm{ex}}_r(n,{\rm Berge}{\text -}T_{\ell_1}\cup (k-1)S_{\ell_2})\leq \frac{(\ell_{\rm max}-1)n}{r-k+2}+O(1).
    \end{eqnarray*}
\end{thm}

In fact, we can prove more in the case where the stars are single edges. 

\begin{thm}\label{thmstarmat}
    \textbf{(i)} If $F$ is not a matching and $r\ge k+|V(F)|$, then for sufficiently large $n$, ${\rm{ex}}_r(n,{\rm Berge}{\text -}F\cup M_{k-1})={\rm{ex}}_r(n,{\rm Berge}{\text -}F)$.

    \textbf{(ii)} If $F$ contains a cycle and $r>k$, then for sufficiently large $n$, ${\rm{ex}}_r(n,{\rm Berge}{\text -}F\cup M_{k-1})={\rm{ex}}_r(n,{\rm Berge}{\text -}F)$.

    \textbf{(iii)} Let $F$ be a graph with $1\leq w\le r-1$ vertices of degree greater than 1 and $r>k+w-1$. Then for sufficiently large $n$, ${\rm{ex}}_r(n,{\rm Berge}{\text -}F\cup M_{k-1})={\rm{ex}}_r(n,{\rm Berge}{\text -}F)$.
\end{thm}


Note that sharp bounds for ${\rm{ex}}_r(n,{\rm Berge}{\text -}S_\ell)$ are known from \cite{B4,D4}. Therefore, the above theorem implies sharp bounds for ${\rm{ex}}_r(n,{\rm Berge}{\text -}S_\ell\cup M_{k-1})$ when $r>k$ and $\ell\geq2$. 

\begin{cor}\label{thmnew12}
    Let integers $\ell\geq2$, $r>k\geq2$ and $n$ be sufficiently large. If $\ell\leq r+1$, then 
    \begin{eqnarray*}
    {\rm{ex}}_r(n,{\rm Berge}{\text -}S_\ell\cup M_{k-1})= \Big\lfloor\frac{n(\ell-1)}{r}\Big\rfloor. 
    \end{eqnarray*}

    If $\ell>r+1$, then 
    \begin{eqnarray*}
    {\rm{ex}}_r(n,{\rm Berge}{\text -}S_\ell\cup M_{k-1})\leq \frac{n}{\ell}\binom{\ell}{r}.
    \end{eqnarray*}
    Moreover, this bound is sharp whenever $\ell$ divides $n$. 
\end{cor}



\subsection{\normalsize Linear forests} 
\ \ \ \ 
We now consider Berge linear forests. A \textit{linear forest} is a graph whose connected components are all paths or isolated vertices. First, we determine the exact value of ${\rm{ex}}^{\rm con}_r(n,{\rm Berge}{\text -}P_{\ell_1}\cup\cdots\cup P_{\ell_k})$ when $r$ is small and all path lengths are odd. 

\begin{prop}\label{prop4.1}
Let $k\geq2$ be an integer, $\ell_1\geq\ell_2\geq\cdots\geq\ell_k\geq1$ be odd integers
and $3\leq r\leq \frac{\sum_{i=1}^{k}(\ell_i+1)}{2}-7$. Then for sufficiently large $n$, 
\begin{eqnarray*}
{\rm{ex}}^{\rm con}_r(n,{\rm Berge}{\text -}P_{\ell_1}\cup\cdots\cup P_{\ell_k})&=&{\rm{ex}}^{\rm con}_r(n,{\rm Berge}{\text -}P_{\sum_{i=1}^{k}(\ell_i+1)-1}) \\
&=&\binom{\frac{\sum_{i=1}^{k}(\ell_i+1)}{2}-1}{r-1}\Big(n-\frac{\sum_{i=1}^{k}(\ell_i+1)}{2}+1\Big)+\binom{\frac{\sum_{i=1}^{k}(\ell_i+1)}{2}-1}{r}.
\end{eqnarray*}
\end{prop}

Combining Theorem \ref{lem2.2}, Theorem \ref{thmstarmat} implies sharp bounds for ${\rm{ex}}_r(n,{\rm Berge}{\text -}P_\ell\cup M_{k-1})$ when $r>k+\ell$ and $\ell>2$. In the case of small $r$, we determine the exact Tur\'{a}n numbers of Berge-$P_{\ell_1}\cup P_{\ell_2}$ for odd integers $\ell_1$ and $\ell_2$, using the above proposition. 

\begin{thm}\label{thmnew14}
Let $r\geq3$ and $n$ be sufficiently large. 
\begin{kst}
\item[\textbf{(i)}] If $\ell$ is an odd integer and $\ell\geq 2r+11$, then 
\begin{eqnarray*}
{\rm{ex}}_r(n,{\rm Berge}{\text -}P_\ell\cup P_{1})=\max\bigg\{{\rm{ex}}_r(n,{\rm Berge}{\text -}P_\ell), \binom{\frac{\ell+1}{2}}{r-1}\Big(n-\frac{\ell+1}{2}\Big)+\binom{\frac{\ell+1}{2}}{r}\bigg\}.
\end{eqnarray*}
\item[\textbf{(ii)}] If $\ell_1,\ell_2$ are odd integers and $\ell_1\geq\ell_2\geq r+6$, then 
\begin{eqnarray*}
{\rm{ex}}_r(n,{\rm Berge}{\text -}P_{\ell_1}\cup P_{\ell_2})=\max\bigg\{{\rm{ex}}_r(n,{\rm Berge}{\text -}P_{\ell_1}), \binom{\frac{\ell_1+\ell_2}{2}}{r-1}\Big(n-\frac{\ell_1+\ell_2}{2}\Big)+\binom{\frac{\ell_1+\ell_2}{2}}{r}\bigg\}.
\end{eqnarray*}
\end{kst}
\end{thm}

When $\ell_1=\ell_2>r$, we have $\binom{\ell_1}{r}\frac{1}{\ell_1}<\binom{\frac{\ell_1+\ell_2}{2}}{r-1}$. 
It follows from Theorem \ref{lem2.2}\,\textbf{(i)} that ${\rm{ex}}_r(n,{\rm Berge}{\text -}P_{\ell_1})< \binom{\frac{\ell_1+\ell_2}{2}}{r-1}\big(n-\frac{\ell_1+\ell_2}{2}\big)+\binom{\frac{\ell_1+\ell_2}{2}}{r}$ for sufficiently large $n$. 
Consequently, the above theorem yields the following result. 

\begin{cor}
Let $\ell$ be an odd integer and $\ell\geq r+6\geq9$. Then for sufficiently large $n$, 
\begin{eqnarray*}
{\rm{ex}}_r(n,{\rm Berge}{\text -}2P_\ell)=\binom{\ell}{r-1}(n-\ell)+\binom{\ell}{r}.
\end{eqnarray*}
\end{cor}

The rest of this paper is organized as follows. In Section 3, we provide the proofs of Theorems \ref{thm1} and \ref{thmstarmat}, as well as the proof of a more general theorem that includes Theorems \ref{thm2} and \ref{thm21}. The proofs of Proposition \ref{prop4.1} and Theorem \ref{thmnew14} are presented in Section 4. 

\section{\normalsize Forests with star components} 
First, let us present a result that will be used in the proof. Khormali and Palmer \cite{D4} proved the following lemma that establish degree conditions for the existence of a Berge-$S_\ell$. 

\begin{lem}[Khormali and Palmer \cite{D4}]\label{lem2.3}
\begin{kst}
\item[\textbf{(i)}] Fix integers $\ell>r\geq2$ and let $H$ be an $r$-graph. If $x$ is a vertex of degree $d(x)>\binom{\ell-1}{r-1}$ in $H$, then $H$ contains a Berge-$S_\ell$ with center $x$.
\item[\textbf{(ii)}] Fix integers $r\geq2$ and $\ell\leq r$ and let $H$ be an $r$-graph. If $x$ is a vertex of degree $d(x)>\ell-1$ in $H$, then $H$ contains a Berge-$S_\ell$ with center $x$.
\end{kst}
\end{lem}



Let $H$ be an $r$-graph and $V'\subseteq V(H)$ be a nonempty subset. Let $H[V']$ denote the subhypergraph of $H$ induced by $V'$. Given a hypergraph $H$ and a vertex $v\in V(H)$, the link hypergraph $H_v$ is defined as
\begin{eqnarray*}
H_v=\big\{e\backslash\{v\}\,\big|\,e\in E(H), v\in e\big\}.
\end{eqnarray*}

Now let us start with the proof of Theorem \ref{thm1} that we restate here for convenience. 

\begin{thm*}
Let integers $k\geq2$, $\ell\geq1$, $1\leq i\leq \ell+1$ and $2\leq r\leq k+\ell-1$. Suppose that ${\rm{ex}}_p(n,{\rm Berge}{\text -}T_\ell)\leq \binom{\ell}{p}\frac{n}{\ell}$ for each $2\le p\le r$, in particular, the Erd\H os-S\'os conjecture holds for $T_\ell$. Then for sufficiently large $n$, 
\begin{eqnarray*}
{\rm{ex}}_r(n,{\rm Berge}{\text -}T_\ell\cup (k-1)S_i)\leq \bigg(\binom{\ell+k-1}{r}-\binom{k-1}{r}\bigg)\bigg\lceil\frac{n-k+1}{\ell}\bigg\rceil+\binom{k-1}{r}. 
\end{eqnarray*}
Moreover, let $0\leq t\leq k-1$. If $\ell\,|\,n-k+1$, then 
\begin{eqnarray*}
{\rm{ex}}_r(n,{\rm Berge}{\text -}T_\ell\cup tS_\ell\cup (k-1-t)S_{\ell+1})= \bigg(\binom{\ell+k-1}{r}-\binom{k-1}{r}\bigg)\frac{n-k+1}{\ell}+\binom{k-1}{r}. 
\end{eqnarray*}
\end{thm*}

\begin{proof}[{\bf Proof}]
For the lower bound, we consider the following $r$-graph $H^*$. Let $A^*$ be a set of $k-1$ vertices and $B^*$ be a set of $n-k+1$ vertices. Partition the vertices of $B^*$ into $\lfloor\frac{n-k+1}{\ell}\rfloor$ classes of size $\ell$ and a class of size $n-k+1-\ell\lfloor\frac{n-k+1}{\ell}\rfloor$. For each partition class $X$ of $B^*$, we form a complete $r$-graph $K_{\ell+k-1}^{(r)}$ or $K_{n-\ell\lfloor\frac{n-k+1}{\ell}\rfloor}^{(r)}$ on the vertex set $A^*\cup X$.

Since the skeleton of any Berge-$T_\ell$ or Berge-$S_\ell$ contains $\ell+1$ vertices, the skeleton of a Berge-$T_\ell$ or Berge-$S_\ell$ in $H^*$ contains at least one vertex of $A^*$. This implies that $H^*$ is Berge-$T_\ell\cup (k-1)S_\ell$-free as $|A^*|<k$, hence $H^*$ is also Berge-$T_\ell\cup tS_\ell\cup (k-1-t)S_{\ell+1}$-free. By the definition of $H^*$, we have
\begin{eqnarray*}
e(H^*)=\bigg(\binom{\ell+k-1}{r}-\binom{k-1}{r}\bigg)\bigg\lfloor\frac{n-k+1}{\ell}\bigg\rfloor+\binom{n-\ell\lfloor\frac{n-k+1}{\ell}\rfloor}{r}.
\end{eqnarray*}

We now continue with the upper bound. Note that $T_\ell\cup (k-1)S_{i}$ is a subgraph of $T_\ell\cup (k-1)S_{\ell+1}$ for any $1\leq i\leq \ell$. 
It is sufficient to consider only the case when $\ell\,|\,n-k+1$. In fact, if $\ell\,\nmid\,n-k+1$, then let $n'=k-1+\ell\lceil\frac{n-k+1}{\ell}\rceil$. Clearly, $\ell\,|\,n'-k+1$. Since ${\rm{ex}}_r(n,{\rm Berge}{\text -}T_\ell\cup (k-1)S_{\ell+1})\leq {\rm{ex}}_r(n',{\rm Berge}{\text -}T_\ell\cup (k-1)S_{\ell+1})$, we only need to show that ${\rm{ex}}_r(n',{\rm Berge}{\text -}T_\ell\cup (k-1)S_{\ell+1})\leq \big(\binom{\ell+k-1}{r}-\binom{k-1}{r}\big)\frac{n'-k+1}{\ell}+\binom{k-1}{r}$.

Let $H$ be an $r$-graph on $n$ vertices with
\begin{eqnarray*}
e(H)>e(H^*)=\bigg(\binom{\ell+k-1}{r}-\binom{k-1}{r}\bigg)\frac{n-k+1}{\ell}+\binom{k-1}{r}.
\end{eqnarray*}
We will show that $H$ contains a ${\rm Berge}{\text -}T_\ell\cup (k-1)S_{\ell+1}$. Let $d:=d(\ell,k,r)$ be a large enough fixed constant, and set $V_0=\big\{v\in V(H)\,\big|\,d_{H}(v)>d\big\}$. 
We proceed by induction on $r$.

Let us consider the case $r=2$. This is essentially a theorem of Fang and Yuan \cite{10-10A4}, but we add the proof for the sake of completeness, and because the proof for larger $r$ follows the same line of thought. If $|V_0|<k-1$, then let $B=V(H)\backslash V_0$. Since $n$ is sufficiently large,
\begin{eqnarray*}
e(H[B])&\geq& e(H)-\binom{|V_0|}{2}-|V_0|(n-|V_0|) \\
&\geq& e(H)-\binom{k-2}{2}-(k-2)(n-k+2) \\
&>& \frac{(\ell-1)n}{2} \\
&\geq& {\rm{ex}}(n,T_\ell),
\end{eqnarray*}
which implies $H[B]$ contains a copy of $T_\ell$, denoted by $T_\ell$. On the other hand,
\begin{eqnarray*}
\sum_{v\in B}d_{H[B]}(v)&=& 2e(H[B]) \\
&\geq& 2\bigg(e(H)-\binom{|V_0|}{2}-|V_0|(n-|V_0|)\bigg) \\
&>& (\ell+2k-3-2|V_0|)n-(\ell+2k-3)(k-1)+2|V_0|^2 \\
&\geq& (\ell+1)n-(\ell+2k-3)(k-1)+2|V_0|^2.
\end{eqnarray*}
So $\frac{\sum_{v\in B}d_{H[B]}(v)}{n-|V_0|}\geq \ell+\epsilon_1$ for some $\epsilon_1>0$.
Let $a_1$ be the number of vertices of degree at most $\ell$ in $B$ within $H[B]$. Then $\sum_{v\in B}d_{H[B]}(v)\leq a_1\ell+(n-|V_0|-a_1)d$. Thus,
\begin{eqnarray*}
a_1\leq \frac{d-\ell-\epsilon_1}{d-\ell}(n-|V_0|)=(1-\epsilon_1')(n-|V_0|),
\end{eqnarray*}
where $\epsilon_1':=\frac{\epsilon_1}{d-\ell}$. Clearly, $0<\epsilon_1'<1$ as $d$ is large enough. So the number of vertices of degree greater than $\ell$ in $B$ within $H[B]$ is at least $\epsilon_1'(n-|V_0|)=\Omega(n)$.
For each vertex $v$ of degree greater than $\ell$ in $B$ within $H[B]$, there is a $S_{\ell+1}$ with center $v$. Then there are $\Omega(n)$ $S_{\ell+1}$ with centers in $B$ within $H[B]$. Note that $T_\ell$ or $S_{\ell+1}$ in $H[B]$ has at most $d(\ell+2)$ neighbors as $d_{H}(v)\leq d$ for any $v\in B$. Since $n$ is sufficiently large, we may find $k-1$ vertex-disjoint copies of $S_{\ell+1}$ in $H[B]$ which are disjoint from $V(T_\ell)$. This implies that there is a copy of $T_\ell\cup (k-1)S_{\ell+1}$ in $H$.

If $|V_0|\geq k-1$, then let $A=\{u_1,\dots,u_{k-1}\}$ be a $(k-1)$-subset of $V_0$ and $B=V(H)\backslash A$. Since $n$ is sufficiently large,
\begin{eqnarray*}
e(H[B])&\geq& e(H)-\binom{k-1}{2}-(k-1)(n-k+1) \\
&>& \frac{(\ell-1)(n-k+1)}{2} \\
&\geq& {\rm{ex}}(n-k+1,T_\ell),
\end{eqnarray*}
which implies $H[B]$ contains a copy of $T_\ell$, denoted by $T_\ell$. Note that for any $u_i\in A$, we have $d_{H}(u_i)>d$. Since $d$ is large enough, there exists a $S_{\ell+1}$ with center $u_1$ in $H$ that is disjoint from both $A\backslash\{u_1\}$ and $V(T_\ell)$. Now suppose that we have identified a $(k-2)S_{\ell+1}$ in $H$ that is disjoint from $\{u_{k-1}\}\cup V(T_\ell)$, where $u_1,\ldots,u_{k-2}$ are the centers of $k-2$ stars, respectively. Since $d$ is large enough, we have $d_{H}(u_{k-1})>d>(k-1)({\ell+2})+\ell$. Thus, there is a $S_{\ell+1}$ with center $u_{k-1}$ in $H$ that is vertex-disjoint from $T_\ell$ and $(k-2)S_{\ell}$. This implies that there is a copy of $T_\ell\cup (k-1)S_{\ell+1}$ in $H$.

Now let $r\geq3$ and assume that the upper bound holds for any $2\leq r'<r$. 

\smallskip
{\bf{Case 1.}} $|V_0|<k-1$.
\smallskip
\smallskip

Let $B=V(H)\backslash V_0$. We may suppose that $H_v$ is a Berge-$T_\ell\cup (k-2)S_{\ell+1}$-free $(r-1)$-graph for any $v\in V_0$. Otherwise, assume that
$H_w$ contains a copy of $(r-1)$-uniform Berge-$T_\ell\cup (k-2)S_{\ell+1}$ for some $w\in V_0$. Since the skeleton of this $(r-1)$-uniform Berge-$T_\ell\cup (k-2)S_{\ell+1}$ does not contain the vertex $w$, $H$ contains a copy of Berge-$T_\ell\cup (k-2)S_{\ell+1}$. Let us remove the $(k-2)(\ell+1)+\ell$ hyperedges of the Berge-$T_\ell\cup (k-2)S_{\ell+1}$ from $H$ and let this resulting hypergraph be $H'$. Since $d$ is large enough, 
by Lemma \ref{lem2.3}, there is a Berge-$S_{(k-1)(\ell+2)+\ell}$ with center $w$ in $H'$. This implies that this Berge-$S_{(k-1)(\ell+2)+\ell}$ is hyperedge-disjoint from the Berge-$T_\ell\cup (k-2)S_{\ell+1}$ in $H$.
Note that the skeleton of the Berge-$T_\ell\cup (k-2)S_{\ell+1}$ contains $(k-2)(\ell+2)+\ell+1$ vertices and does not contain the vertex $w$. Therefore, there is a Berge-$S_{\ell+1}$ in $H'$ whose skeleton is disjoint from the skeleton of the Berge-$T_\ell\cup (k-2)S_{\ell+1}$, as at most $(k-2)(\ell+2)+\ell+1$ vertices of the skeleton of this Berge-$S_{(k-1)(\ell+2)+\ell}$ are shared with the skeleton of the Berge-$T_\ell\cup (k-2)S_{\ell+1}$. Thus, $H$ contains a copy of Berge-$T_\ell\cup (k-1)S_{\ell+1}$ and we are done. 

Note that $|V(H_u)|\leq n-1$ for any $u\in V_0$. According to the inductive assumption, we have
\begin{eqnarray*}
d_{H}(u)=e(H_u)&\leq& \bigg(\binom{\ell+k-2}{r-1}-\binom{k-2}{r-1}\bigg)\frac{|V(H_u)|-k+2}{\ell}+\binom{k-2}{r-1} \notag \\
&\leq& \bigg(\binom{\ell+k-2}{r-1}-\binom{k-2}{r-1}\bigg)\frac{n-k+1}{\ell}+\binom{k-2}{r-1}.
\end{eqnarray*}
Thus
\begin{eqnarray*}
&{}&\sum_{v\in V(H)}d_{H}(v)= \sum_{v\in V_0}d_{H}(v)+\sum_{v\in B}d_{H}(v) \notag \\
&{}&\leq |V_0|\cdot\bigg[\bigg(\binom{\ell+k-2}{r-1}-\binom{k-2}{r-1}\bigg)\frac{n-k+1}{\ell}+\binom{k-2}{r-1}\bigg]+(n-|V_0|)d_H(B),
\end{eqnarray*}
where $d_H(B)$ denote the average degree of the vertices in $B$.

Recall that $e(H)>e(H^*)$ and $\sum_{v\in V(H)}d_{H}(v)=r\cdot e(H)$ for any $r$-graph $H$. Then
\begin{eqnarray*}
&{}&\sum_{v\in V(H)}d_{H}(v)=r\cdot e(H)>r\cdot e(H^*)=\sum_{v\in V(H^*)}d_{H^*}(v)= \sum_{v\in A^*}d_{H^*}(v)+\sum_{v\in B^*}d_{H^*}(v) \notag \\
&{}&= (k-1)\bigg[\bigg(\binom{\ell+k-2}{r-1}-\binom{k-2}{r-1}\bigg)\frac{n-k+1}{\ell}+\binom{k-2}{r-1}\bigg]+(n-k+1)\binom{\ell+k-2}{r-1}.
\end{eqnarray*}

Since $n$ is sufficiently large and $|V_0|<k-1$, by comparing the coefficients of $n$ in the above two inequalities, we get $d_{H}(B)>\binom{\ell+k-2}{r-1}$. This implies that $d_{H}(B)\geq \binom{\ell+k-2}{r-1}+\epsilon_2$ for some constant $\epsilon_2>0$. Therefore,
\begin{eqnarray}
\sum_{v\in B}d_{H}(v)\geq (n-|V_0|)\bigg(\binom{\ell+k-2}{r-1}+\epsilon_2\bigg).
\end{eqnarray}

Let $a_2$ be the number of vertices of degree at most $\binom{\ell+k-2}{r-1}$ in $B$. Then
\begin{eqnarray}
\sum_{v\in B}d_{H}(v)\leq a_2\binom{\ell+k-2}{r-1}+(n-|V_0|-a_2)d.
\end{eqnarray}
Combining (3.1) and (3.2) and solving for $a_2$, we obtain
\begin{eqnarray*}
a_2\leq \frac{d-\binom{\ell+k-2}{r-1}-\epsilon_2}{d-\binom{\ell+k-2}{r-1}}(n-|V_0|)=(1-\epsilon_2')(n-|V_0|),
\end{eqnarray*}
where $\epsilon_2':=\frac{\epsilon_2}{d-\binom{\ell+k-2}{r-1}}$. Clearly, $0<\epsilon_2'<1$ as $d$ is large enough. So the number of vertices of degree greater than $\binom{\ell+k-2}{r-1}$ in $B$ is $\epsilon_2'(n-|V_0|)=\Omega(n)$.

For each vertex $v$ of degree greater than $\binom{\ell+k-2}{r-1}$ in $B$, there is a Berge-$S_{\ell+k-1}$ with center $v$ by Lemma \ref{lem2.3}. As $|V_0|<k-1$, there is a Berge-$S_{\ell+1}$ with center $v$ in $H$ whose skeleton is disjoint from $V_0$. Thus, there are $\Omega(n)$ Berge-$S_{\ell+1}$ with centers in $B$, each of whose skeletons is disjoint from $V_0$. Since $n$ is sufficiently large and $d_{H}(v)\leq d$ for any $v\in B$, we may find $k-1$ hyperedge-disjoint copies of Berge-$S_{\ell+1}$ in $H$ whose skeletons are all in $B$ and are pairwise disjoint. This implies that there is a Berge-$(k-1)S_{\ell+1}$ in $H$ whose skeleton is in $B$.
Now we delete the $(k-1)(\ell+2)$ vertices from the skeleton, as well as at most $(k-1)(\ell+2)d$ associated hyperedges. Denote by $H''$ this resulting $r$-graph. Then
$$e(H'')>\bigg(\binom{\ell+k-1}{r}-\binom{k-1}{r}\bigg)\frac{n-k+1}{\ell}+\binom{k-1}{r}-(k-1)(\ell+2)d.$$

Since $\binom{\ell+k-1}{r}-\binom{k-1}{r}>\binom{\ell}{r}$ and $n$ is sufficiently large, by assumption, we have 
\begin{eqnarray*}
e(H'')> \binom{\ell}{r}\frac{n}{\ell}\geq {\rm{ex}}_r(n,{\rm Berge}{\text -}T_\ell),
\end{eqnarray*}
which implies $H''$ contains a copy of Berge-$T_\ell$. Thus, by the definition of $H''$, there is a Berge-$T_\ell\cup (k-1)S_{\ell+1}$ in $H$.

\smallskip
{\bf{Case 2.}} $|V_0|\geq k-1$.
\smallskip
\smallskip

Let $A=\{u_1,\dots,u_{k-1}\}$ be a $(k-1)$-subset of $V_0$ and $B=V(H)\backslash A$. We now distinguish two cases.

\smallskip
{\bf{Subcase 2.1.}} $e(H[B])\leq e(H^*[B])$.
\smallskip
\smallskip

For $1\leq i\leq \min\{r,k-1\}$, let $E_i$ be the set of hyperedges of $H$ intersecting $A$ in exactly $i$ vertices and $E^*_i$ be the set of hyperedges of $H^*$ intersecting $A^*$ in exactly $i$ vertices. By the definition of $H^*$, observe that when $k-1\geq r$, $H^*[A^*]$ is a complete $r$-graph on $k-1$ vertices. When $k-1\geq r-1$, each vertex of $B^*$ is contained in a hyperedge with every $(r-1)$-subset of $A^*$. Therefore, we obtain $|E_r|\leq|E^*_r|$ if $k-1\geq r$, and $|E_{r-1}|\leq|E^*_{r-1}|$ if $k-1\geq r-1$. Meanwhile, we have $E_r=E^*_r=\emptyset$ if $k-1<r$, and $E_{r-1}=E^*_{r-1}=\emptyset$ if $k-1<r-1$. Since
\begin{eqnarray*}
e(H[B])+\sum_{i=1}^{\min\{r,k-1\}}|E_i|=e(H)>e(H^*)=e(H^*[B^*])+\sum_{i=1}^{\min\{r,k-1\}}|E^*_i|,
\end{eqnarray*}
there exists some $1\leq j\leq r-2$ such that $|E_j|>|E^*_j|$.

Define the multi-set $E_j(r-j)=\{e\backslash A\,|\,e\in E_j\}$. Clearly, $|E_j(r-j)|=|E_j|$. Note that the members of $E_j(r-j)$ have a multiplicity of at most $\binom{k-1}{j}$. Let $H(r-j)$ denote the $(r-j)$-graph obtained by removing all but one copies of the repeated hyperedges from $E_j(r-j)$. By the definition of $H^*$, we have $|E^*_j|=\binom{k-1}{j}\frac{n-k+1}{\ell}\binom{\ell}{r-j}$. Then
\begin{eqnarray*}
e(H(r-j))&\geq& \frac{|E_j(r-j)|}{\binom{k-1}{j}}=\frac{|E_j|}{\binom{k-1}{j}} \\
&>& \frac{|E^*_j|}{\binom{k-1}{j}}=\frac{n-k+1}{\ell}\binom{\ell}{r-j} \\
&\geq& {\rm{ex}}_{r-j}(n-k+1,{\rm Berge}{\text -}T_\ell)
\end{eqnarray*}
by assumption. This implies that there is an $(r-j)$-uniform Berge-$T_\ell$ in $H(r-j)$ on $B$. Since each hyperedge of $H(r-j)$ is contained in a hyperedge of $H$, there is a Berge-$T_\ell$ in $H$ whose skeleton is contained in $B$.

Note that for any $u_i\in A$, we have $d_{H}(u_i)>d$. Since $d$ is large enough, there exists a Berge-$S_{2\ell+1}$ with center $u_1$ in $H$ that is hyperedge-disjoint from the Berge-$T_\ell$ and whose skeleton is disjoint from $A\backslash\{u_1\}$. 
Now suppose that we have identified a Berge-$(k-2)S_{2\ell+2}$ in $H$ that is hyperedge-disjoint from the Berge-$T_\ell$ and whose skeleton intersects $A$ at the set of vertices $\{u_1,\ldots,u_{k-2}\}$, where $u_1,\ldots,u_{k-2}$ are the centers of $k-2$ stars in the skeleton, respectively.
Let us remove the $\ell+(k-2)(2\ell+2)$ hyperedges of the Berge-$T_\ell$ and Berge-$(k-2)S_{2\ell+2}$ from $H$, and denote by $H'''$ this resulting hypergraph. Since $d$ is large enough, by Lemma \ref{lem2.3}, there is a Berge-$S_{(k-2)(2\ell+3)+3\ell+3}$ with center $u_{k-1}$ in $H'''$. 
Note that the skeleton of the Berge-$(k-2)S_{2\ell+2}$ contains $(k-2)(2\ell+3)$ vertices
and does not contain the vertex $u_{k-1}$. Therefore, there is a Berge-$S_{2\ell+2}$ in $H'''$ whose skeleton is disjoint from the skeletons of the Berge-$T_\ell$ and Berge-$(k-2)S_{2\ell+2}$. By the definition of $H'''$, there is a Berge-$(k-1)S_{2\ell+2}$ in $H$ that is hyperedge-disjoint from the Berge-$T_\ell$ and whose skeleton is disjoint from the skeleton of the Berge-$T_\ell$. Since at most $\ell+1$ vertices of each star in the skeleton of this Berge-$(k-1)S_{2\ell+2}$ are shared with the skeletons of the Berge-$T_\ell$, there is a Berge-$T_\ell\cup (k-1)S_{\ell+1}$ in $H$.

\smallskip
{\bf{Subcase 2.2.}} $e(H[B])>e(H^*[B])$.
\smallskip
\smallskip

By the definition of $H^*$, we have $e(H^*[B])=\frac{n-k+1}{\ell}\binom{\ell}{r}$. Then
\begin{eqnarray*}
e(H[B])&>& e(H^*[B]) \\
&=& \frac{n-k+1}{\ell}\binom{\ell}{r} \\
&\geq& {\rm{ex}}_{r}(n-k+1,{\rm Berge}{\text -}T_\ell)
\end{eqnarray*}
by assumption. This implies that there is a Berge-$T_\ell$ in $H[B]$. As in Subcase 2.1, the degree condition on the vertices in $A$ guarantees the existence of a Berge-$(k-1)S_{\ell+1}$ that together with this Berge-$T_\ell$ forms a Berge-$T_\ell\cup (k-1)S_{\ell+1}$ in $H$.
This completes the proof.
\end{proof}


Now let us continue with the proofs of Theorem \ref{thm2} and Theorem \ref{thm21}. We will prove the following more general theorem. 

\begin{thm}
    Let $k\geq2$, $\ell_1,\ell_2\geq2$, $r\geq\ell_1+\ell_2+k-1$ and $T$ be a tree with $\ell_1$ edges. Let $\ell_{\rm max}:=\max\{\ell_1,\ell_2\}$ and $\ell_{\rm min}:=\min\{\ell_1,\ell_2\}$. Assume that ${\rm{ex}}_r(n,{\rm Berge}{\text -}T)\le \frac{n(\ell_1-1)}{r+1}$. Then $(\ell_{\rm min}-1)\big\lfloor\frac{n-k+1}{r-k+2}\big\rfloor\leq {\rm{ex}}_r(n,{\rm Berge}{\text -}T\cup (k-1)S_{\ell_2})\leq \frac{(\ell_{\rm max}-1)n}{r-k+2}+O(1)$.
\end{thm}

Theorem \ref{thm2} follows by applying \textbf{(ii)} of Theorem \ref{lem2.2}, while Theorem \ref{thm21} follows by applying a theorem of Gy\H{o}ri, Salia, Tompkins and Zamora \cite{10-10A2}, who showed that for any  $\ell$-edge tree $T_{\ell}\neq S_{\ell}$, we have ${\rm{ex}}_r(n,{\rm Berge}{\text -}T_{\ell})\leq \frac{n(\ell-1)}{r+1}$ provided
$r\geq \ell(\ell-2)$. Let us remark that they conjecture that the same holds for $r\ge \ell$.

\begin{proof}[{\bf Proof}]
For the lower bound, we consider the following $r$-graph $\hat{H}$. Let $A$ be a set of $k-1$ vertices and $B$ be a set of $n-k+1$ vertices. Partition the vertices of $B$ into $\lfloor\frac{n-k+1}{r-k+2}\rfloor$ classes of size $r-k+2$ and a class of size $n-k+1-(r-k+2)\lfloor\frac{n-k+1}{r-k+2}\rfloor$. For each partition class $X$ of size $r-k+2$ of $B$, we take $\ell_{\rm min}-1$ $(r-k+1)$-uniform hyperedges (this is possible as $\binom{r-k+2}{r-k+1}=r-k+2>\ell_{\rm min}$). Add the set $A$ to each hyperedge to form an $r$-graph on $n$ vertices. 
Clearly, $|\hat{H}|=(\ell_{\rm min}-1)\lfloor\frac{n-k+1}{r-k+2}\rfloor$.

Now let us show that $\hat{H}$ is Berge-$T\cup (k-1)S_{\ell_2}$-free. Since the degree of each vertex in $B$ is at most $\ell_{\rm min}-1$, the skeleton of any Berge-$S_{\ell_2}$ in $\hat{H}$ uses at least one vertex from $A$. 
Note that for any partition class $X$ of size $r-k+2$ of $B$, there are only $\ell_{\rm min}-1$ hyperedges in $X\cup A$. Therefore, any Berge-$T$ in $\hat{H}$ intersects with at least two classes from $B$. Then there exist two adjacent hyperedges $e_1$ and $e_2$ in Berge-$T$ such that $e_1\cap e_2=A$. 
This implies that the skeleton of any Berge-$T$ in $\hat{H}$ use at least one vertex from $A$. Thus, $\hat{H}$ is Berge-$T\cup (k-1)S_{\ell_2}$-free as $|A|=k-1$.

We now continue with the upper bound. Suppose 
that $H$ is a Berge-$T\cup (k-1)S_{\ell_2}$-free $r$-graph on $n$ vertices. Let $v_1,\dots,v_n$ be the vertices of $H$ listed in decreasing order of their degrees. By deleting $v_1,\dots,v_{k-1}$ from each hyperedge of $H$, we obtain a hypergraph on $n-k+1$ vertices in which every hyperedge has at least $r-k+1$ vertices. 

We now remove some vertices from the hyperedges of the hypergraph such that each hyperedge contains exactly $r-k+1$ vertices. First, we order the hyperedges of $H$ with less than $k-1$ vertices from the set $\{v_1,\dots,v_{k-1}\}$ arbitrarily. Then we go through them in this order, and each time we take an unused $(r-k+1)$-subset avoiding $\{v_1,\dots,v_{k-1}\}$. If we are unable to do so for a hyperedge, it means that every $(r-k+1)$-subset of the hyperedge has already been taken. The hyperedge has at least $r-k+2$ vertices outside $\{v_1,\dots,v_{k-1}\}$, thus there are at least $\binom{r-k+2}{r-k+1}> \ell_1$ such $(r-k+1)$-sets, and they form an $(r-k+1)$-uniform Berge-$T$, which will later be shown to yield a contradiction. Therefore, we obtain an $(r-k+1)$-graph $H^{(r-k+1)}$ without repeated hyperedges.
If $H^{(r-k+1)}$ is Berge-$T$-free, then by our assumption on $T$, we have 
\begin{eqnarray*}
e(H)=e(H^{(r-k+1)})\leq \frac{(\ell_1-1)(n-k+1)}{r-k+2}\leq \frac{(\ell_{\rm max}-1)n}{r-k+2}+O(1).
\end{eqnarray*}

Now suppose that $H^{(r-k+1)}$ is the resulting $(r-k+1)$-graph that may contain repeated hyperedges, and that $H^{(r-k+1)}$ contains a Berge-$T$. 
By the definition of $H^{(r-k+1)}$, we may find a Berge-$T$ in $H$ whose skeleton is disjoint from $\{v_1,\dots,v_{k-1}\}$, denoted by $P$. Let $V_0$ be the set of vertices of its skeleton. Clearly, $|V_0|=\ell_1+1$. Let $H'=\big(V(H), E(H)\backslash E(P)\big)$. Then $e(H')=e(H)-\ell_1$. 

Let $D:=D(\ell_1,\ell_2,k,r)$ be a large enough fixed constant, and set
\begin{eqnarray*}
V_1&=& \big\{v\ \big|\ d_{H'}(v)>D,\ v\in V(H)\big\}; \\
V_2&=& V(H)\backslash V_1.
\end{eqnarray*}

\begin{cla}\label{clam1}
$|V_1|\leq k-2$.
\end{cla}
\begin{proof}[{\bf{Proof of Claim.}}]
Suppose to the contrary that $|V_1|\geq k-1$. Then $v_1,\dots,v_{k-1}\in V_1$. Let $V'_1:=\{v_1,\ldots,v_{k-1}\}$. Note that for any $v_i\in V'_1$, we have $d_{H'}(v_i)>D$. Since $D$ is large enough, there exists a Berge-$S_{\ell_1+\ell_2+1}$ with center $v_1\in V'_1$ in $H'$ and its skeleton is disjoint from $V'_1\backslash\{v_1\}$. Now suppose that we have identified a Berge-$(k-2)S_{\ell_1+\ell_2+1}$ in $H'$ and its skeleton intersects $V'_1$ at the set of vertices $\{v_1,\ldots,v_{k-2}\}$, where $v_1,\ldots,v_{k-2}$ are the centers of $k-2$ stars in the skeleton, respectively. 
Let us remove the $(k-2)(\ell_1+\ell_2+1)$ hyperedges of the Berge-$(k-2)S_{\ell_1+\ell_2+1}$ from $H'$ and let this resulting hypergraph be $H''$.
Since $D$ is large enough, we have $$d_{H'}(v_{k-1})>D> (k-1)(\ell_1+\ell_2+2)+(k-2)({\ell_1+\ell_2+1}).$$
By Lemma \ref{lem2.3}\,\textbf{(ii)}, there is a Berge-$S_{(k-1)(\ell_1+\ell_2+2)-1}$ with center $v_{k-1}$ in $H''$. This implies that this Berge-$S_{(k-1)(\ell_1+\ell_2+2)-1}$ is hyperedge-disjoint from the Berge-$(k-2)S_{\ell_1+\ell_2+1}$ in $H'$.
Note that the skeleton of the Berge-$(k-2)S_{\ell_1+\ell_2+1}$ contains $(k-2)(\ell_1+\ell_2+2)$ vertices and does not
contain the vertex $v_{k-1}$. Therefore, there is a Berge-$S_{\ell_1+\ell_2+1}$ in $H''$ whose skeleton is disjoint from the skeleton of the Berge-$(k-2)S_{\ell_1+\ell_2+1}$, as at most $(k-2)(\ell_1+\ell_2+2)$ vertices of the skeleton of this Berge-$S_{(k-1)(\ell_1+\ell_2+2)-1}$ are shared with the skeleton of the Berge-$(k-2)S_{\ell_1+\ell_2+1}$. Thus, by the definition of $H''$, there is a Berge-$(k-1)S_{\ell_1+\ell_2+1}$ in $H'$ and its skeleton contains $V'_1$, where $v_1,\ldots,v_{k-1}$ are the centers of $k-1$ stars in the skeleton. Recall that $V_0$ is the set of vertices of skeleton of $P$. 
Since at most $\ell_1+1$ vertices of each star in the skeleton of this Berge-$(k-1)S_{\ell_1+\ell_2+1}$ are shared with $V_0$, there is a Berge-$(k-1)S_{\ell_2}$ in $H'$ whose skeleton is disjoint from $V_0$. So by the definition of $H'$, there is a Berge-$T\cup (k-1)S_{\ell_2}$ in $H$, which is a contradiction. 
\end{proof}

By Claim \ref{clam1}, we obtain $V_0\subseteq V_2$. Let $V'_2=V_2\backslash V_0$ and $d_{H'}(V'_2)$ denote the average degree of the vertices in $V'_2$ within $H'$. Now let us estimate $d_{H'}(V'_2)$. 

\begin{cla}\label{clam2}
If $|V_1|=k-2$, then $d_{H'}(V'_2)\leq \ell_2-1$. If $|V_1|<k-2$, then $d_{H'}(V'_2)\leq \ell_2-1+\epsilon$ for any $\epsilon>0$.
\end{cla}
\begin{proof}[{\bf{Proof of Claim.}}]
First let us consider the case when $|V_1|=k-2$ and suppose to the contrary that $d_{H'}(V'_2)>\ell_2-1$. Then there exists a vertex $u\in V'_2$ such that $d_{H'}(u)\geq \ell_2$. Let us remove vertices from the hyperedges incident to $u$ to ensure the resulting hyperedges are disjoint from $V_0\cup V_1$ and have size exactly $\ell_2$ (this is possible as $r\geq \ell_1+\ell_2+k-1$). By Lemma \ref{lem2.3}\,\textbf{(ii)}, there exists an $\ell_2$-uniform Berge-$S_{\ell_2}$ centered at $u$ within these $\ell_2$-uniform hyperedges, with its skeleton contained in $V'_2$. This implies that there exists an $r$-uniform Berge-$S_{\ell_2}$ centered at $u$ within $H'$, with its skeleton contained in $V'_2$.
As $D$ is large enough and $d_{H'}(v)>D$ for any $v\in V_1$, we may construct a Berge-$(k-2)S_{\ell_1+\ell_2+1}$ in $H'$ whose hyperedges and skeleton are disjoint from those of the $r$-uniform Berge-$S_{\ell_2}$. Since at most $\ell_1+1$ vertices of each star in the skeleton of this Berge-$(k-2)S_{\ell_1+\ell_2+1}$ are shared with $V_0$, there is a Berge-$(k-2)S_{\ell_2}$ in $H'$ whose skeleton is disjoint from $V_0$. Thus, there is a Berge-$(k-1)S_{\ell_2}$ in $H'$ whose skeleton is disjoint from $V_0$. By the definition of $H'$, there is a Berge-$T\cup (k-1)S_{\ell_2}$ in $H$, which is a contradiction.

Now consider the case when $|V_1|<k-2$ and suppose to the contrary that $d_{H'}(V'_2)\geq\ell_2-1+\epsilon$ for some fixed constant $0<\epsilon<1$. Then
\begin{eqnarray}
\sum_{v\in V'_2}d_{H'}(v)\geq (n-|V_1|-\ell_1-1)(\ell_2-1+\epsilon).
\end{eqnarray}

Let $b$ be the number of vertices of degree at most $\ell_2-1$ in $V'_2$ within $H'$. Then
\begin{eqnarray}
\sum_{v\in V_2}d_{H'}(v)\leq b(\ell_2-1)+(n-|V_1|-\ell_1-1-b)D.
\end{eqnarray}
Combining (3.4) and (3.5) and solving for $b$, we obtain
\begin{eqnarray*}
b\leq \frac{D-\ell_2+1-\epsilon}{D-\ell_2+1}(n-|V_1|-\ell_1-1)=(1-\epsilon')(n-|V_1|-\ell_1-1),
\end{eqnarray*}
where $\epsilon':=\frac{\epsilon}{D-\ell_2+1}$. Clearly, $0<\epsilon'<1$ as $D$ is large enough. So the number of vertices of degree greater than $\ell_2-1$ in $V'_2$ within $H'$ is at least $\epsilon'(n-|V_1|-\ell_1-1)=\Omega(n)$.

For each vertex $v$ of degree greater than $\ell_2-1$ in $V'_2$ within $H'$, let us remove vertices from the hyperedges incident to $v$ to ensure the resulting hyperedges are disjoint from $V_0\cup V_1$ and have size exactly $\ell_2$ (this is possible as $r\geq \ell_1+\ell_2+k-1>\ell_1+\ell_2+1+|V_1|$). By Lemma \ref{lem2.3}\,\textbf{(ii)}, there exists an $\ell_2$-uniform Berge-$S_{\ell_2}$ centered at $v$ within these $\ell_2$-uniform hyperedges, with its skeleton contained in $V'_2$. This implies that there exists an $r$-uniform Berge-$S_{\ell_2}$ centered at $v$ within $H'$, with its skeleton contained in $V'_2$. Thus, there are $\Omega(n)$ Berge-$S_{\ell_2}$ with centers in $V'_2$ within $H'$, each of whose skeletons contained in $V'_2$.
Since $n$ is sufficiently large and $d_{H'}(v)<D$ for any $v\in V'_2$, we may find $k-1$ hyperedge-disjoint copies of Berge-$S_{\ell_2}$ in $H'$ whose skeletons are all in $V'_2$ and are pairwise disjoint. This implies that there is a Berge-$(k-1)S_{\ell_2}$ in $H'$ whose skeleton is in $V'_2$. So by the definition of $H'$, there is a Berge-$T\cup (k-1)S_{\ell_2}$ in $H$, which is a contradiction.
\end{proof}

Note that $\sum_{v\in V(G)}d_{G}(v)=r\cdot e(G)$ for any $r$-graph $G$. Since $d_{H'}(v)\leq e(H')$ for any $v\in V_1$, we have
\begin{eqnarray*}
r\cdot e(H')&=&\sum_{v\in V_1}d_{H'}(v)+\sum_{v\in V_0}d_{H'}(v)+\sum_{v\in V'_2}d_{H'}(v) \\
&\leq& |V_1|\cdot e(H')+(\ell_1+1)D+(n-|V_1|-\ell_1-1)d_{H'}(V'_2).
\end{eqnarray*}
Therefore,
\begin{eqnarray*}
e(H')\leq \frac{d_{H'}(V_2)(n-|V_1|-\ell_1-1)+(\ell_1+1)D}{r-|V_1|}.
\end{eqnarray*}

Since $n$ is sufficiently large, by Claim \ref{clam2}, we may choose $\epsilon$ small enough such that
\begin{eqnarray*}
\frac{d_{H'}(V_2)(n-|V_1|-\ell_1-1)+(\ell_1+1)D}{r-|V_1|}\leq \frac{(\ell_2-1)(n-k-\ell_1+1)+(\ell_1+1)D}{r-k+2}.
\end{eqnarray*}
Then
\begin{eqnarray*}
e(H)=e(H')+\ell_1\leq \frac{(\ell_2-1)n}{r-k+2}+O(1)\leq \frac{(\ell_{\rm max}-1)n}{r-k+2}+O(1).
\end{eqnarray*}
This completes the proof.
\end{proof}

In \cite{D4}, Khormali and Palmer completely determined the Tur\'{a}n number of Berge matchings for sufficiently large $n$.

\begin{thm}[Khormali and Palmer \cite{D4}]\label{lem4.1}
Fix integers $k \geq 1$ and $r \geq 2$. Then for $n$ large enough,
\begin{eqnarray*}
{\rm{ex}}_r(n,{\rm Berge}{\text -}M_k)=
\begin{cases}
k - 1, & {\rm{if}}\ r \geq 2k - 1; \\
\binom{2k-1}{r}, & {\rm{if}}\ k < r < 2k - 1; \\
n - k + 1, & {\rm{if}}\ r = k; \\
\binom{k-1}{r-1} (n - k + 1) + \binom{k-1}{r}, & {\rm{if}}\ r \leq k - 1.
\end{cases}
\end{eqnarray*}
\end{thm}

We now begin with the proof of Theorem \ref{thmstarmat} that we restate here for convenience.

\begin{thm*}
    \textbf{(i)} If $F$ is not a matching and $r\ge k+|V(F)|$, then for sufficiently large $n$, ${\rm{ex}}_r(n,{\rm Berge}{\text -}F\cup M_{k-1})={\rm{ex}}_r(n,{\rm Berge}{\text -}F)$.

    \textbf{(ii)} If $F$ contains a cycle and $r>k$, then for sufficiently large $n$, ${\rm{ex}}_r(n,{\rm Berge}{\text -}F\cup M_{k-1})={\rm{ex}}_r(n,{\rm Berge}{\text -}F)$.

    \textbf{(iii)} Let $F$ be a graph with $1\leq w\le r-1$ vertices of degree greater than 1 and $r>k+w-1$. Then for sufficiently large $n$, ${\rm{ex}}_r(n,{\rm Berge}{\text -}F\cup M_{k-1})={\rm{ex}}_r(n,{\rm Berge}{\text -}F)$.
\end{thm*}

\begin{proof}[\bf Proof]
    We prove statements \textbf{(i)} and \textbf{(ii)} together. Let us assume indirectly that there is a Berge-$F\cup M_{k-1}$-free $r$-graph $H$ with more than ${\rm{ex}}_r(n,{\rm Berge}{\text -}F)$ hyperedges. Then $H$ contains a Berge-$F$ denoted by $BF$. Let $U$ denote the set of vertices in the skeleton of $BF$. Let us delete the hyperedges of $BF$ from $H$ to obtain a subhypergraph $H'$. 
    
    If $H'$ contains a Berge-$M_{k-1+|V(F)|}$, then $H'$ contains a Berge-$M_{k-1}$ whose skeleton does not intersect $U$. This Berge matching together with $BF$ forms a Berge-$F\cup M_{k-1}$ in $H$, which is a contradiction. Therefore, $H$ has at most ${\rm{ex}}_r(n,{\rm Berge}{\text -}M_{k-1+|V(F)|})+e(F)$ hyperedges. If $r\ge k+|V(F)|$, then this is $O(1)$ by Theorem \ref{lem4.1}. Since $P_2\subseteq F$, we have ${\rm{ex}}_r(n,{\rm Berge}{\text -}F)\geq{\rm{ex}}_r(n,{\rm Berge}{\text -}P_2)=\Omega(n)$, a contradiction completing the proof of \textbf{(i)}. 
    


    If $k<r<k+|V(F)|$, then ${\rm{ex}}_r(n,{\rm Berge}{\text -}M_{k-1+|V(F)|})+e(F)\le \binom{k+|V(F)|-1}{r-1}(n-k+1)+\binom{k-1}{r}+e(F)$.
    This is smaller than ${\rm{ex}}_r(n,{\rm Berge}{\text -}F)$ if $F$ contains a cycle. Indeed, a theorem of Ellis and Linial \cite{elli} implies that ${\rm{ex}}_r(n,{\rm Berge}{\text -}C_t)=\omega(n)$ for every $t$. This completes the proof of \textbf{(ii)}. 



    
    It is left to prove \textbf{(iii)}. Let $BM$ denote a largest Berge matching in $H$ and $M$ denote its skeleton. Then $e(BM)<k-1+e(F)+|V(F)|$, and thus $|V(M)|\le 2(k-2+e(F)+|V(F)|)$. Otherwise, 
    assume that $e(BM)\geq k-1+e(F)+|V(F)|$. Observe that the skeleton of $BF$ shares at most $|V(F)|$ vertices with $V(M)$, and at most $e(F)$ of its hyperedges are contained in $E(BM)$. Since $e(BM)\geq k-1+e(F)+|V(F)|$, there is a subgraph Berge-$M_{k-1}$ of $BM$ that is hyperedge-disjoint from $BF$ and whose skeleton is disjoint from the skeleton of $BF$. This Berge-$M_{k-1}$ together with $BF$ forms a Berge-$F\cup M_{k-1}$ in $H$, a contradiction. 

    Now let us remove the hyperedges of $BM$ from $H$, and denote this resulting hypergraph by $H''$. We claim that each hyperedge of $H''$ has at least $r-1$ vertices in $V(M)$.  Otherwise, there exists a larger Berge matching in $H$, contradicting the maximality of $BM$. Therefore, at least one of the $(r-1)$-sets $A$ in $V(M)$ is contained in at least \[\frac{{\rm{ex}}_r(n,{\rm Berge}{\text -}F)-e(BM)-\binom{2(k-2+e(F)+|V(F)|)}{r-1}}{\binom{2(k-2+e(F)+|V(F)|)}{r-1}}=\Omega(n)\] hyperedges of $H''$. Let $H'''$ denote subhypergraph of $H''$ consisting of the hyperedges that contain $A$ and are not contained in $V(M)$. Let $B$ denote the set of vertices that form a hyperedge of $H'''$ with $A$. 
    
    Now we embed $F$ into $H'''$ the following way. We embed the vertices of $F$ with degree greater than 1 into $A$ and the rest of the vertices into $B$. For each edge of the skeleton incident to a vertex of degree 1, one endpoint is embedded into $A$, and the other into $v\in B$. Then the corresponding hyperedge is $A\cup \{v\}$. The rest of the edges are inside $A$, we use arbitrary unused hyperedges of $H'''$ for them. We can choose a distinct one each time, since there are $\Omega(n)$ such hyperedges. Therefore, we embedded a Berge-$F$ such a way that it contains $w$ vertices of $V(M)$. If there are at least $k-1$ edges of $M$ avoiding the vertices of $F$, then the corresponding hyperedges of $BM$ together with the newly embedded Berge-$F$ form the forbidden configuration, a contradiction that completes the proof. If there are less than $k-1$ edges of $M$ avoiding the vertices of $F$, then $M$ has less than $k-1+w$ edges, thus $H$ is Berge-$M_{k-1+w}$-free, hence has $O(1)$ hyperedges by Theorem \ref{lem4.1}, a contradiction. This completes the proof. 
\end{proof}

\section{\normalsize Linear forests} 

Let us start with the proof of Proposition \ref{prop4.1} that we restate here for convenience. 

\begin{prop*}
Let $k\geq2$ be an integer, $\ell_1\geq\ell_2\geq\cdots\geq\ell_k\geq1$ be odd integers and $3\leq r\leq \frac{\sum_{i=1}^{k}(\ell_i+1)}{2}-7$. Then for sufficiently large $n$, 
\begin{eqnarray*}
{\rm{ex}}^{\rm con}_r(n,{\rm Berge}{\text -}P_{\ell_1}\cup\cdots\cup P_{\ell_k})&=&{\rm{ex}}^{\rm con}_r(n,{\rm Berge}{\text -}P_{\sum_{i=1}^{k}(\ell_i+1)-1}) \\
&=&\binom{\frac{\sum_{i=1}^{k}(\ell_i+1)}{2}-1}{r-1}\Big(n-\frac{\sum_{i=1}^{k}(\ell_i+1)}{2}+1\Big)+\binom{\frac{\sum_{i=1}^{k}(\ell_i+1)}{2}-1}{r}.
\end{eqnarray*}
\end{prop*}

\begin{proof}[\bf Proof] 
Obviously, $P_{\ell_1}\cup\cdots\cup P_{\ell_k}$ is a subgraph of $P_{\sum_{i=1}^{k}(\ell_i+1)-1}$. Hence ${\rm{ex}}^{\rm con}_r(n,{\rm Berge}{\text -}P_{\ell_1}\cup\cdots\cup P_{\ell_k})\leq {\rm{ex}}^{\rm con}_r(n,{\rm Berge}{\text -}P_{\sum_{i=1}^{k}(\ell_i+1)-1})$. Since $\sum_{i=1}^{k}(\ell_i+1)-1$ is odd, by Theorem \ref{lem4.3}, we obtain the upper bound. 
For the lower bound, we consider the following $r$-graph $\tilde{H}$. Let $\tilde{A}$ be a set of $\sum_{i=1}^{k}(\ell_i+1)/2-1$ vertices and $\tilde{B}$ be the set of the remaining $n-\sum_{i=1}^{k}(\ell_i+1)/2+1$ vertices. For each vertex $u$ in $\tilde{B}$, we form a complete $r$-graph $K_{\sum_{i=1}^{k}(\ell_i+1)/2}^{(r)}$ on the vertex set $\tilde{A}\cup\{u\}$. 
By the definition of $\tilde{H}$, the skeleton of a Berge-$P_{\ell_i}$ in $\tilde{H}$ contains at least $\frac{\ell_i+1}{2}$ vertices of $\tilde{A}$. 
This implies that $\tilde{H}$ is Berge-$P_{\ell_1}\cup\cdots\cup P_{\ell_k}$-free as $|\tilde{A}|< \sum_{i=1}^{k}(\ell_i+1)/2$. Furthermore, $\tilde{H}$ is connected and 
\begin{eqnarray*}
e(\tilde{H})= \binom{\frac{\sum_{i=1}^{k}(\ell_i+1)}{2}-1}{r-1}\Big(n-\frac{\sum_{i=1}^{k}(\ell_i+1)}{2}+1\Big)+\binom{\frac{\sum_{i=1}^{k}(\ell_i+1)}{2}-1}{r}.
\end{eqnarray*}
This completes the proof. 
\end{proof}

Next, we proceed with the proof of Theorem \ref{thmnew14} that we restate here for convenience. 

\begin{thm*}
Let $r\geq3$ and $n$ be sufficiently large. 
\begin{kst}
\item[\textbf{(i)}] If $\ell$ is an odd integer and $\ell\geq 2r+11$, then 
\begin{eqnarray*}
{\rm{ex}}_r(n,{\rm Berge}{\text -}P_\ell\cup P_{1})=\max\bigg\{{\rm{ex}}_r(n,{\rm Berge}{\text -}P_\ell), \binom{\frac{\ell+1}{2}}{r-1}\Big(n-\frac{\ell+1}{2}\Big)+\binom{\frac{\ell+1}{2}}{r}\bigg\}.
\end{eqnarray*}
\item[\textbf{(ii)}] If $\ell_1,\ell_2$ are odd integers and $\ell_1\geq\ell_2\geq r+6$, then 
\begin{eqnarray*}
{\rm{ex}}_r(n,{\rm Berge}{\text -}P_{\ell_1}\cup P_{\ell_2})=\max\bigg\{{\rm{ex}}_r(n,{\rm Berge}{\text -}P_{\ell_1}), \binom{\frac{\ell_1+\ell_2}{2}}{r-1}\Big(n-\frac{\ell_1+\ell_2}{2}\Big)+\binom{\frac{\ell_1+\ell_2}{2}}{r}\bigg\}.
\end{eqnarray*}
\end{kst}
\end{thm*}

\begin{proof}[\bf Proof] 
\textbf{(i).} For the lower bound, we consider the $r$-graph $\tilde{H}$ in the proof of Proposition \ref{prop4.1}, where we set $\ell_1=\ell$ and $\ell_2=1$, and denote the $r$-graph by $\tilde{H}_{\ell,1}$. Then $\tilde{H}_{\ell,1}$ is Berge-$P_{\ell}\cup P_{1}$-free. Meanwhile, we have ${\rm{ex}}_r(n,{\rm Berge}{\text -}P_{\ell})\leq {\rm{ex}}_r(n,{\rm Berge}{\text -}P_\ell\cup P_{1})$ as $P_{\ell}$ is a subgraph of $P_{\ell}\cup P_{1}$. Thus, ${\rm{ex}}_r(n,{\rm Berge}{\text -}P_\ell\cup P_{1})\geq \max\big\{{\rm{ex}}_r(n,{\rm Berge}{\text -}P_\ell), \binom{\frac{\ell+1}{2}}{r-1}\big(n-\frac{\ell+1}{2}\big)+\binom{\frac{\ell+1}{2}}{r}\big\}$.

We now continue with the upper bound. Suppose to the contrary that $H$ is a Berge-$P_{\ell}\cup P_{1}$-free $r$-graph on $n$ vertices with
\begin{eqnarray}
e(H)> \max\bigg\{{\rm{ex}}_r(n,{\rm Berge}{\text -}P_\ell), \binom{\frac{\ell+1}{2}}{r-1}\Big(n-\frac{\ell+1}{2}\Big)+\binom{\frac{\ell+1}{2}}{r}\bigg\}.
\end{eqnarray}
We claim that $H$ is disconnected. Otherwise, assume that $H$ is connected. By Proposition \ref{prop4.1}, we have
\begin{eqnarray*}
e(H)\leq {\rm{ex}}^{\rm con}_r(n,{\rm Berge}{\text -}P_{\ell}\cup P_{1})= \binom{\frac{\ell+1}{2}}{r-1}\Big(n-\frac{\ell+1}{2}\Big)+\binom{\frac{\ell+1}{2}}{r},
\end{eqnarray*}
which is a contradiction to (4.1). 
Since $e(H)>{\rm{ex}}_r(n,{\rm Berge}{\text -}P_\ell)$, $H$ contains a Berge-$P_\ell$. 

Let $H_1$ be the connected component that contains a Berge-$P_\ell$, and $H_2=H[V(H)\backslash V(H_1)]$. Note that $H=H_1\cup H_2$ and $H_2$ may not be connected. From the fact that $H$ is Berge-$P_{\ell}\cup P_{1}$-free, it follows that $H_1$ is Berge-$P_{\ell}\cup P_{1}$-free and $H_2$ is Berge-$P_1$-free. Therefore, $e(H_1)\leq {\rm{ex}}^{\rm con}_r(|V(H_1)|,{\rm Berge}{\text -}P_\ell\cup P_{1})$ and $e(H_2)\leq {\rm{ex}}_r(n-|V(H_1)|,{\rm Berge}{\text -}P_{1})=0$. 

If $|V(H_1)|=O(1)$, then $e(H_1)\leq \binom{|V(H_1)|}{r}=O(1)$ and therefore $e(H)=e(H_1)+e(H_2)=O(1)$, which is a contradiction to (4.1). Then $|V(H_1)|$ is large enough as $n$ is sufficiently large. By Proposition \ref{prop4.1}, ${\rm{ex}}^{\rm con}_r(|V(H_1)|,{\rm Berge}{\text -}P_\ell\cup P_{1})\leq \binom{\frac{\ell+1}{2}}{r-1}\big(|V(H_1)|-\frac{\ell+1}{2}\big)+\binom{\frac{\ell+1}{2}}{r}$. Thus, 
\begin{eqnarray*}
e(H)=e(H_1)+e(H_2)&\leq& {\rm{ex}}^{\rm con}_r(|V(H_1)|,{\rm Berge}{\text -}P_\ell\cup P_{1}) \\
&=& \binom{\frac{\ell+1}{2}}{r-1}\Big(|V(H_1)|-\frac{\ell+1}{2}\Big)+\binom{\frac{\ell+1}{2}}{r} \\
&\leq& \binom{\frac{\ell+1}{2}}{r-1}\Big(n-\frac{\ell+1}{2}\Big)+\binom{\frac{\ell+1}{2}}{r},
\end{eqnarray*}
which is a contradiction to (4.1), completing the proof. 

\textbf{(ii).} For the lower bound, we also consider the $r$-graph $\tilde{H}$, where we set $k=2$ and denote the $r$-graph by $\tilde{H}_{\ell_1,\ell_2}$. Then $\tilde{H}_{\ell_1,\ell_2}$ is Berge-$P_{\ell_1}\cup P_{\ell_2}$-free. Meanwhile, we have ${\rm{ex}}_r(n,{\rm Berge}{\text -}P_{\ell_1})\leq {\rm{ex}}_r(n,{\rm Berge}{\text -}P_{\ell_1}\cup P_{\ell_2})$ as $P_{\ell_1}$ is a subgraph of $P_{\ell_1}\cup P_{\ell_2}$. Thus, ${\rm{ex}}_r(n,{\rm Berge}{\text -}P_{\ell_1}\cup P_{\ell_2})\geq \max\big\{{\rm{ex}}_r(n,{\rm Berge}{\text -}P_{\ell_1}), \binom{\frac{\ell_1+\ell_2}{2}}{r-1}\big(n-\frac{\ell_1+\ell_2}{2}\big)+\binom{\frac{\ell_1+\ell_2}{2}}{r}\big\}$.

We now continue with the upper bound. Suppose to the contrary that $H$ is a Berge-$P_{\ell_1}\cup P_{\ell_2}$-free $r$-graph on $n$ vertices with
\begin{eqnarray}
e(H)> \max\bigg\{{\rm{ex}}_r(n,{\rm Berge}{\text -}P_{\ell_1}), \binom{\frac{\ell_1+\ell_2}{2}}{r-1}\Big(n-\frac{\ell_1+\ell_2}{2}\Big)+\binom{\frac{\ell_1+\ell_2}{2}}{r}\bigg\}.
\end{eqnarray}
We claim that $H$ is disconnected. Otherwise, assume that $H$ is connected. By Proposition \ref{prop4.1}, we have
\begin{eqnarray*}
e(H)\leq {\rm{ex}}^{\rm con}_r(n,{\rm Berge}{\text -}P_{\ell_1}\cup P_{\ell_2})= \binom{\frac{\ell_1+\ell_2}{2}}{r-1}\Big(n-\frac{\ell_1+\ell_2}{2}\Big)+\binom{\frac{\ell_1+\ell_2}{2}}{r},
\end{eqnarray*}
which is a contradiction to (4.2).
Since $e(H)>{\rm{ex}}_r(n,{\rm Berge}{\text -}P_{\ell_1})$, $H$ contains a Berge-$P_{\ell_1}$.

Let $H_1$ be the connected component that contains a Berge-$P_{\ell_1}$, and $H_2=H[V(H)\backslash V(H_1)]$. Note that $H=H_1\cup H_2$ and $H_2$ may not be connected. Let $|V(H_1)|=n_1$. Then $|V(H_2)|=n-n_1$. From the fact that $H$ is Berge-$P_{\ell_1}\cup P_{\ell_2}$-free, it follows that $H_1$ is Berge-$P_{\ell_1}\cup P_{\ell_2}$-free and $H_2$ is Berge-$P_{\ell_2}$-free. Therefore, $e(H_1)\leq {\rm{ex}}^{\rm con}_r(n_1,{\rm Berge}{\text -}P_{\ell_1}\cup P_{\ell_2})$ and 
$$e(H_2)\leq {\rm{ex}}_r(n-n_1,{\rm Berge}{\text -}P_{\ell_2})\leq \binom{\ell_2}{r}\frac{n-n_1}{\ell_2}$$ 
by Theorem \ref{lem2.2}\,\textbf{(i)}. 

If $n_1=O(1)$, then $e(H_1)\leq \binom{n_1}{r}=O(1)$. Note that $\binom{\ell_2}{r}\frac{1}{\ell_2}<\binom{\frac{\ell_1+\ell_2}{2}}{r-1}$ as $\ell_1\geq\ell_2$. Then 
\begin{eqnarray*}
e(H)&=& e(H_1)+e(H_2) \\
&\leq& O(1)+\binom{\ell_2}{r}\frac{n-n_1}{\ell_2} \\
&<& \binom{\frac{\ell_1+\ell_2}{2}}{r-1}\Big(n-\frac{\ell_1+\ell_2}{2}\Big)+\binom{\frac{\ell_1+\ell_2}{2}}{r},
\end{eqnarray*}
which is a contradiction to (4.2). Then $n_1$ is large enough as $n$ is sufficiently large. By Proposition \ref{prop4.1}, ${\rm{ex}}^{\rm con}_r(n_1,{\rm Berge}{\text -}P_{\ell_1}\cup P_{\ell_2})\leq \binom{\frac{\ell_1+\ell_2}{2}}{r-1}\Big(n_1-\frac{\ell_1+\ell_2}{2}\Big)+\binom{\frac{\ell_1+\ell_2}{2}}{r}$. Thus,
\begin{eqnarray*}
e(H)=e(H_1)+e(H_2)&\leq& {\rm{ex}}^{\rm con}_r(n_1,{\rm Berge}{\text -}P_{\ell_1}\cup P_{\ell_2})+{\rm{ex}}_r(n-n_1,{\rm Berge}{\text -}P_{\ell_2}) \\
&\leq& \binom{\frac{\ell_1+\ell_2}{2}}{r-1}\Big(n_1-\frac{\ell_1+\ell_2}{2}\Big)+\binom{\frac{\ell_1+\ell_2}{2}}{r}+\binom{\ell_2}{r}\frac{n-n_1}{\ell_2} \\
&<& \binom{\frac{\ell_1+\ell_2}{2}}{r-1}\Big(n_1-\frac{\ell_1+\ell_2}{2}\Big)+\binom{\frac{\ell_1+\ell_2}{2}}{r}+\binom{\frac{\ell_1+\ell_2}{2}}{r-1}(n-n_1) \\
&=& \binom{\frac{\ell_1+\ell_2}{2}}{r-1}\Big(n-\frac{\ell_1+\ell_2}{2}\Big)+\binom{\frac{\ell_1+\ell_2}{2}}{r},
\end{eqnarray*}
which is a contradiction to (4.2). This completes the proof. 
\end{proof}

\bigskip
\textbf{Funding}: 
The research of Zhou is supported by the National Natural Science Foundation of China (Nos. 11871040, 12271337, 12371347) and the China Scholarship Council (No. 202406890088).

The research of Gerbner is supported by the National Research, Development and Innovation Office - NKFIH under the grant KKP-133819.

The research of Yuan is supported by the National Natural Science Foundation of China (Nos. 11871040, 12271337, 12371347).

\bigskip
\noindent{\bf{Declaration of interest}}

The authors declare no known conflicts of interest.

\bigskip
\noindent{\bf{Acknowledgements}}

We would like to thank Hilal Hama Karim for several helpful discussions on this problem.

\end{document}